\documentclass[10pt]{amsart}

\usepackage[utf8]{inputenc}
\usepackage[all]{xy}
\usepackage{amsmath,amssymb}
\usepackage{amssymb}
\usepackage{algorithmicx}
\usepackage[english]{babel}
\usepackage{graphics}
\usepackage{graphicx}
\graphicspath{ {images/} }
\usepackage{epstopdf}
\usepackage{amsmath}
\usepackage{amsfonts}
\usepackage{amsthm} 
\usepackage{mathrsfs}
\usepackage{tikz-cd}
\usepackage{tikz}

\usetikzlibrary{shapes, shadows, arrows}
\usepackage[noadjust]{cite}

\makeatletter
\def\blfootnote{\gdef\@thefnmark{}\@footnotetext}
\makeatother

\usepackage{epsfig}
\usepackage{graphics}
\usepackage{dcpic, pictexwd}
\usepackage{float}

\usepackage{mathtools}

\usepackage{subcaption}
\usepackage{adjustbox}
\usepackage{xcolor}
\theoremstyle{plain}
\usepackage{amsthm}
\usepackage{thmtools}
\usepackage{float}
\usepackage{todonotes}

\newtheorem*{theorem*}{Theorem}

\newtheorem{theorem}{Theorem}[section]

\newtheorem{lemma}[theorem]{Lemma}
\newtheorem{proposition}[theorem]{Proposition}

\theoremstyle{remark}
\newtheorem{remark}[theorem]{Remark}

\theoremstyle{definition}
\newtheorem{definition}[theorem]{Definition}

\definecolor{darkgreen}{rgb}{0.0, 0.5, 0.0}

\usepackage[colorlinks=true,
                    linkcolor=blue,
                    urlcolor=blue,
                    citecolor=blue,
                    anchorcolor=blue]{hyperref}

 \def\Z{{\mathbb{Z}}}

\def\mod{{\rm Map}}
\def\pmod{{\rm PMap}}

\def\Ends{{\rm Ends}}

 \def\Sym{{\rm Sym}}

\begin{document}
\blfootnote{\textup{2000} \textit{Mathematics Subject Classification}:57M07, 20F05, 20F38}

 \title[Small Sets of Topological Generators for Big Mapping Class Groups] {Small Sets of Topological Generators for Big Mapping Class Groups}

\author[T{\"{u}}l\.{i}n Altun{\"{o}}z, Celal Can Bellek, Em{\.{i}}r G{\"{u}}l,       Mehmetc\.{i}k Pamuk, and O\u{g}uz Y{\i}ld{\i}z ]{T{\"{u}}l\.{i}n Altun{\"{o}}z, Celal Can Bellek, Em{\.{i}}r G{\"{u}}l,      Mehmetc\.{i}k Pamuk, and O\u{g}uz Y{\i}ld{\i}z}
\address{Faculty of Engineering, Ba\c{s}kent University, Ankara, Turkey} 
\email{tulinaltunoz@baskent.edu.tr} 
\address{Department of Mathematics, Middle East Technical University,
 Ankara, Turkey}
\email{celal.bellek@metu.edu.tr}
\address{Department of Mathematics, Middle East Technical University,
 Ankara, Turkey}
\email{gul.emir@metu.edu.tr} 
\address{Department of Mathematics, Middle East Technical University,
 Ankara, Turkey}
 \email{mpamuk@metu.edu.tr}
 \address{Department of Mathematics, Middle East Technical University,
 Ankara, Turkey}
  \email{oguzyildiz16@gmail.com}

\begin{abstract}  
Let $S(n)$ be the infinite-type surface with infinite genus and $n \in 
\mathbb{N}$ ends, all of which are accumulated by genus. The mapping class group of this surface, $\mathrm{Map}(S(n))$, is a Polish group that is not countably generated, but it is countably topologically generated. This paper focuses on finding minimal sets of generators for $\mathrm{Map}(S(n))$. We show that for $n \ge 8$, $\mathrm{Map}(S(n))$ is topologically generated by three elements, and for $n \ge 3$, it is topologically generated by four elements. We also establish a generating set of two elements for the Loch Ness Monster surface $S(1)$, and a generating set of three elements for the Jacob's Ladder surface $S(2)$.
\end{abstract}

\maketitle
\setcounter{secnumdepth}{2}
\setcounter{section}{0}

\section{Introduction}

The mapping class group of a surface $S$, denoted $\mod(S)$, is the group of isotopy classes of orientation-preserving self-homeomorphisms of $S$ that fix the boundary components (if any) pointwise and the ends setwise.  It is a fundamental object in low-dimensional topology that encodes the symmetries of the surface. For surfaces of finite type, those with a finitely generated fundamental group, the structure of this group is well-understood. A classical result shows that it is finitely generated by Dehn twists~\cite{Dehn1987, Humphries1979, Lickorish1964}. Research has since focused on finding minimal generating sets, culminating in the proof by Wajnryb, later refined by Korkmaz, that $\mod(S)$ can be generated by just two elements for a closed surface of genus $g \ge 2$~\cite{Korkmaz2004}.

A particularly fruitful line of inquiry has been the generation of mapping class groups by torsion elements, especially involutions. McCarthy and Papadopoulos first showed that for genus $g \ge 3$, $\mod(S)$ is generated by infinitely many conjugates of a single involution~\cite{McCarthy1987}. Luo demonstrated that a finite set of involutions is sufficient~\cite{Luo2000}, leading to a series of improvements. Brendle and Farb found a generating set of six involutions for $g \ge 3$~\cite{Brendle2004}, which was subsequently reduced to four for $g \ge 7$ by Kassabov~\cite{Kassabov2003}, and finally to three for $g \ge 6$ by Korkmaz~\cite{Korkmaz2020} and the fifth author~\cite{Yildiz2022}.

In recent years, attention has shifted to mapping class groups of infinite-type surfaces, 
often called big mapping class groups, see the overview of Aramayona-Vlamis \cite{Aramayona_2020} and the algebraic and topological analysis of Patel-Vlamis \cite{Patel2018}. More recently, Baik investigated the topological \emph{normal} generation of these groups, establishing conditions under which the normal closure of a single element generates the group \cite{baik2024topological}. These groups exhibit richer and more complex behavior. They are not finitely generated; however, when endowed with the compact-open topology, they become Polish groups. A key property of such groups is that they are countably topologically generated, meaning a dense subgroup can be countably generated.

Prior work on generating sets for the specific family $S(n)$ established bounds for finite topological generating sets (e.g. at most seven involutions~\cite{huynh}, later improved to five involutions~\cite{apyinvolution}) which we further sharpen here by relaxing the requirement that they be involutions, instead allowing for infinite order generators.
These results bring the known bounds for generating big mapping class groups much closer to those established for their finite-type counterparts and represent a significant step toward understanding the minimal generating sets for these complex groups.

The notation in this paper follows standard conventions in mapping class groups, with a few key specifics. The surface of infinite genus with $n$ ends, each accumulated by genus, is denoted by $S(n)$, and its mapping class group by $\mod(S(n))$. For simplicity, we abuse notation by denoting a diffeomorphism and its isotopy class with the same symbol. Group composition, $f \circ g$, is written concisely as $fg$.  The right-handed Dehn twist about a simple closed curve $a$ (i.e., $t_a$) is represented by the corresponding capital letter, such as $A$. In the context of the surfaces $S(n)$ with $n \ge 3$ ends, a system of double indices is used to specify both the genus component and the end component of the curve.  Specific Dehn twists are denoted by indexed capital letters, like $A_i^j$, $B_i^j$, and $C_{i-1}^j$, corresponding to curves $a_i^j$, $b_i^j$, and $c_{i-1}^j$ . The lower index, starting at $i=1, 2, 3, \dots$, primarily corresponds to the genus component or position along the infinite chain of genera (position of genera in one end in Figure~\ref{fig:inf}) .  The upper index, running from $j=1, 2, \dots, n$, specifies the end component or which of the $n$ accumulated ends the curve is near (the vertical or rotational position around the ends in Figure~\ref{fig:inf}). The inverse $X^{-1}$ of any mapping class $X$ is denoted by $\overline{X}$. Additionally, a homeomorphism with infinite support, known as a handle shift, is typically denoted by $h$ or $h_{i,j}$.  When a simplified notation is used in a proof (e.g., Theorem~\ref{thm:first} for $n \ge 8$), the indices may be reduced, where $A_{i}^{j}$, $B_{i}^{j}$, and $C_{i-1}^{j}$ are denoted as $A_j$, $B_j$, and $C_j$ respectively, with $i=1, 2, 3, \dots$ and $j=1, 2, \dots, n$. For example, $B_1^1$ may be denoted as $B_1$, $C_0^3$ as $C_3$, and $A_1^4$ as $A_4$, when the context is clear. 
\begin{figure}[H]
      \centering
      \includegraphics[width=0.5\textwidth]{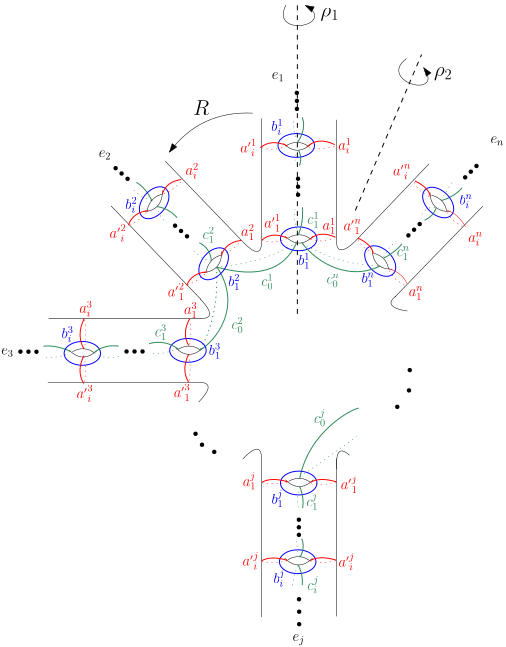}
      \caption{A diagram of the surface $S(n)$, an infinite-type surface with $n$ ends accumulated by genus, showing the standard system of curves used for generating sets. The index $j$ corresponds to the end, while $i$ corresponds to the genus level.}
      \label{fig:inf}
   \end{figure}

\subsection{Main Results}

In this paper, we investigate minimal topological generating sets for the mapping class group of a specific family of infinite-type surfaces, $S(n)$, which has infinite genus and $n$ ends, each accumulated by genus. For these groups, previous work established that $\mod(S(n))$ can be topologically generated by at most seven elements for $n \ge 3$, a bound that was later improved to five for $n \ge 6$~\cite{apyinvolution}.

The main contribution of this paper is to further reduce the number of required topological generators for $\mod(S(n))$. We establish the following results:
\begin{itemize}
     \item For $n\geq8$, $\mod(S(n))$ is topologically generated by three elements (Theorem~\ref{thm:first}).
     \item For $n \ge 3$, $\mod(S(n))$ is topologically generated by four elements (Theorem~\ref{thm:second}).
     \item For the Jacob's Ladder surface ($n=2$), $\mod(S)$ is topologically generated by three elements (Theorem~\ref{thm:third}).
     \item For the Loch Ness Monster surface ($n=1$),
     $\mod(S)$ is topologically generated by two elements (Theorem~\ref{thm:fourth}).
\end{itemize}

\begin{remark}
We suspect that the requirement for a fourth generator in the range $3 \le n \le 7$ is an artifact of our specific construction rather than an intrinsic algebraic property of the groups. Our proof for $n \ge 8$ relies on a generator $F_1$ supported disjointly on ends with indices $\{1, 3, 4, n-1, n\}$. When $3 \leq n \leq 7$, some intermediate elements obtained through the operations we use (conjugation by rotations) fail to stay disjointly supported and hence cannot be use to isolate Dehn twists. We conjecture that a construction using different generators, or an alternative strategy that avoids these specific issues would demonstrate that three elements suffice for all $n\geq 2$.
\end{remark}

The paper is structured as follows. In Section $2$, we establish the necessary background on infinite-type surfaces, their classification, and the key elements of their mapping class groups, such as Dehn twists and handle shifts. Finally, Section $3$ is dedicated to the proofs of our main theorems, giving explicit topological generating sets for $S(n)$ with $n\geq8$, $n\ge3$, as well as for the special cases of the Jacob's Ladder ($n=2$) and the Loch Ness Monster ($n=1$) surfaces.

\section{Preliminaries on infinite-type surfaces}
\subsection{Classification of infinite-type surfaces}
To classify surfaces of infinite type we use the \textit{space of ends} $\Ends(S)$, which records the distinct directions to infinity of the surface.  The construction begins with exiting sequences (nested connected open sets with compact boundary that eventually avoid every compact subset of $S$); $\Ends(S)$ is the set of equivalence classes of such sequences, equipped with the topology generated by the sets $U^*$, where $U\subset S$ is open with compact boundary and $U^*$ consists of those ends represented by exiting sequences eventually contained in $U$.  Intuitively, the space of ends describes how many different directions goes to infinity, how those directions relate to each other and whether those directions contain infinitely many genera or not. We say that an end is accumulated by genus if every element of the sequence defining the end contains infinitely many genera. The classification theorem for orientable infinite-type surfaces then asserts that two such surfaces are homeomorphic exactly when they have the same genus and number of boundary components and there exists a homeomorphism $\Ends(S_1)\cong\Ends(S_2)$. We follow the definitions and conventions used by Aramayona-Vlamis~\cite{Aramayona_2020}.

\begin{theorem}
    Let $S_1$ and $S_2$ be two infinite type surfaces, and $b_1$, $b_2$ be the number of boundary components and $g_1$, $g_2$ be the number of genus of these surfaces respectively. Then, $S_1 \cong S_2$ if and only if $g_1 = g_2$, $b_1=b_2$, and there is a homeomorphism
    \begin{align*}
        \Ends(S_1) \rightarrow \Ends(S_2)
    \end{align*}
    that restricts to a homeomorphism between their respective subspaces consisting of ends accumulated by genus.
\end{theorem}
\subsection{Generating the big mapping class groups}
\subsection*{Pure mapping class group of an infinite-type surface} 
\begin{definition}
    \textit{The pure mapping class group}, denoted by $\pmod(S(n))$, is the subgroup of $\mod(S(n))$ such that it fixes $\Ends(S)$ pointwise.
\end{definition}

For the surfaces of infinite-type, $\mod(S(n))$ is not countably generated.  However, since it is a quotient of the group of orientation-preserving self-homeomorphisms of $S(n)$ (equipped with the compact-open topology), $\mod(S(n))$ inherits a topology. Because of this, $\mod(S(n))$ is a Polish group~\cite{Aramayona_2020}, meaning in particular that it is separable.  Therefore, $\mod(S(n))$ is topologically generated by a countable set.
We have the following exact sequence:
\begin{align*}
    1 \rightarrow \pmod(S(n)) \rightarrow \mod(S(n)) \rightarrow \Sym_n \rightarrow 1.
\end{align*}
Here, $\Sym_n$ is the symmetric group on $n$ letters and the last map is the projection defined by the action of a mapping class on the space of ends, which is the symmetric group on $n$ letters for $\mod(S(n))$. It follows that $\mod(S(n))$ is topologically generated by the generators of $\pmod(S(n))$ together with mapping classes whose image in $\Sym_n$ generate it.

\subsection*{Handle shifts}
The generators of these groups often include not only Dehn twists, but also homeomorphisms with infinite support called \textit{handle shifts}, as shown by Patel and Vlamis~\cite{Patel2018}.

Following~\cite{tez}, we define the handle shift as follows:
Consider the surface $\mathbb{R} \times [-1,1]$ with disks of radius $1/4$ removed and a copy of $S_1^1$ attached along the boundaries of the removed disks at every point $(n,0)$ where $n \in \mathbb{Z}$. This surface is called the \textit{the model surface of a handle shift} and denote it by $\Sigma$.

Note that $\Sigma$ is a surface with two ends accumulated by genus that correspond to $\pm \infty$ of $\mathbb{R}$ and two disjoint boundary components $\mathbb{R}\times\{-1\}$ and $\mathbb{R}\times\{1\}$. We call the end corresponding to $-\infty$ the repelling end and the one corresponding to $\infty$ the attracting end. We can embed $\Sigma$ to any infinite-type surface S with at least two ends accumulated by genus. We define $h: \Sigma \rightarrow \Sigma$ as
\begin{align*}
    h(x,y)=
    \begin{cases}
        (x+1,y) & \text{if }  y \in [-\dfrac{1}{2},\dfrac{1}{2}],\\
        (x+2-2y,y) & \text{if }  y \in [\dfrac{1}{2},1],\\
        (x+2+2y,y) & \text{if } y \in [-1, -\dfrac{1}{2}]
    \end{cases}
\end{align*}
on $\mathbb{R} \times [-1,1]$. This self-homeomorphism $h$ is called a \textit{handle shift}.

\begin{figure}[H]
      \centering
      \includegraphics[width=0.60\textwidth]{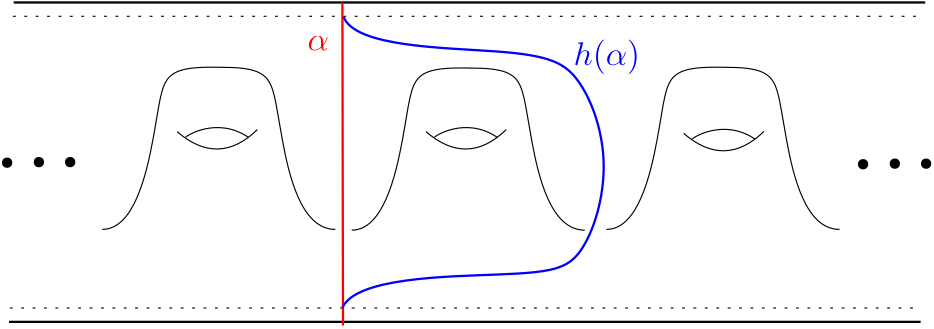}
      \caption{The action of a handle shift $h$ on a transverse curve $\alpha$. The model surface $\Sigma$ illustrates the shift of genera from one region to another}
      \label{fig:handleshift}
\end{figure}

Patel and Vlamis showed in their initial paper that for infinite-type surfaces with more than one end accumulated by genus, handle shifts and Dehn twists are required to topologically generate $\pmod(S)$ [Proposition 6.3, \citenum{Patel2018}]. Moreover Aramayona-Patel-Vlamis improved this result by proving that $\pmod(S)$ can be split as a semi-direct product of $\overline{\pmod_\mathrm{c}(S)}$ and a product of handle shifts \cite{arapatevla}. We state this result for the case relevant to us in this paper:
\begin{theorem}[{\cite[Corollary 6]{arapatevla}}]
    For $S(n)$,
    \[
    \pmod(S(n)) = \overline{\pmod_\mathrm{c}(S(n))} \rtimes \Z^{n-1}.
    \]
\end{theorem}

This result shows that any set that topologically generates $\overline{\pmod_\mathrm{c}(S(n))}$ and that contains $n-1$ handle shifts with different attracting and repelling ends topologically generates the entire pure mapping class group.

\subsection{Special infinite-type surfaces}
\subsection*{The Loch Ness Monster surface} The closed surface with one end accumulated by genus is called the Loch Ness Monster Surface.
\begin{figure}[H]
      \centering
      \includegraphics[width=0.4\textwidth]{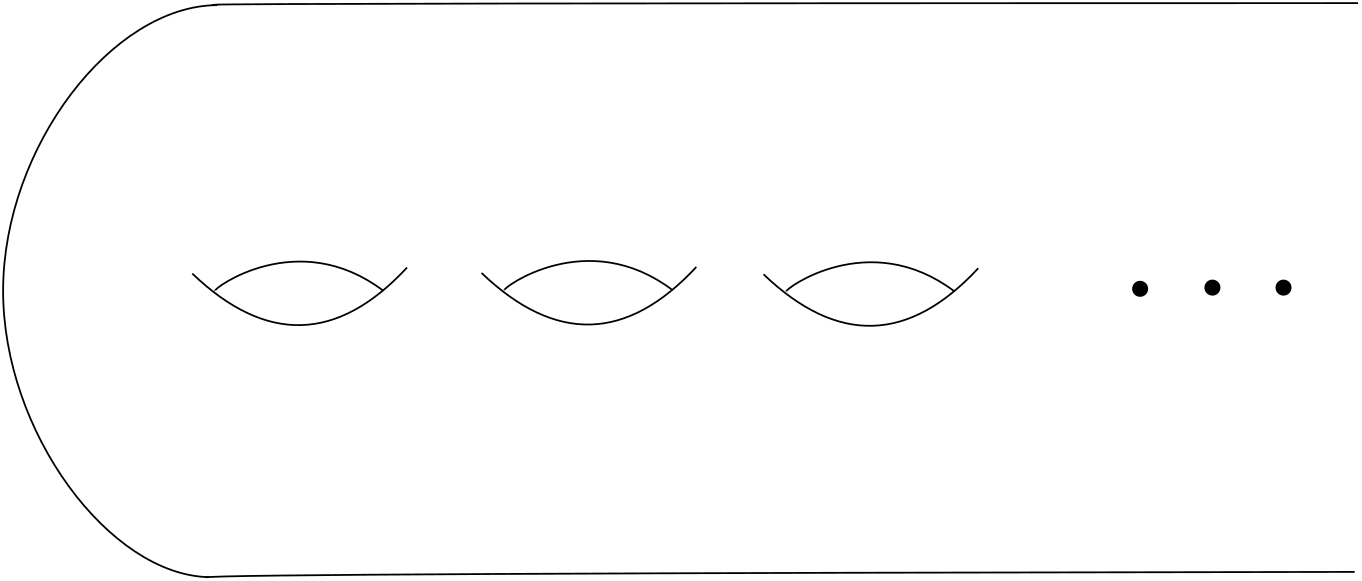}
      \caption{An embedding of the Loch Ness Monster surface, the infinite-genus surface with a single end.}
      \label{fig:model}
   \end{figure}

\subsection*{The Jacob's Ladder surface}\label{Jacob} The closed surface with two end accumulated by genus is called The Jacob's Ladder Surface.

\begin{figure}[H]
      \centering
      \includegraphics[width=0.4\textwidth]{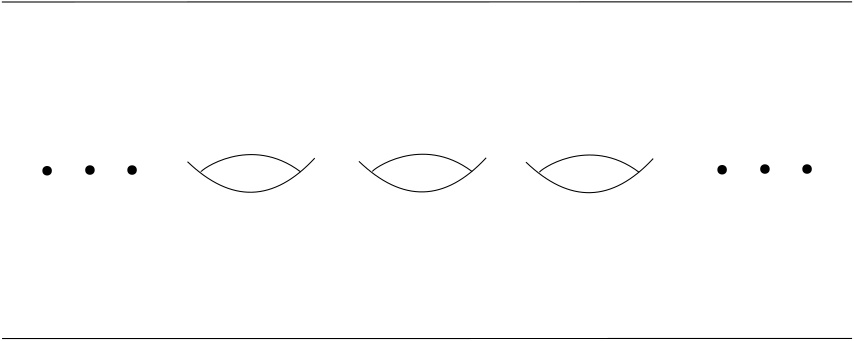}
      \caption{An embedding of the Jacob's Ladder surface, the infinite-genus surface with two ends}
      \label{Jacob's Ladder}
\end{figure}

\vskip 0.1cm
\noindent{\bf Acknowledgements.} 
 This work is supported by the Scientific and Technological Research Council of Turkey (TUBİTAK) [grant number 125F253]

\section{Minimal set of generators}
\subsection{Surfaces with more than two ends}
\vspace{2mm}

Our main strategy is to show that our generating sets contain certain elements in $\mod(S(n))$ that are enough to topologically generate $\mod(S(n))$. We will modify the following two lemmas from \cite{apyinvolution}.

\begin{lemma}[{\cite[Lemma~2.2]{apyinvolution}}]\label{lemm2.2apy}
    For $n \geq 3$, the group topologically generated by the elements
\begin{align*}
    \{\rho_1, \rho_2, A_1^1\overline{A_1^2}, B_1^1\overline{B_1^2}, C_0^1\overline{C_0^2}, h_{1,2}\}
\end{align*}
contains the Dehn twists $A_1^2\overline{A_2^2}$, $B_1^2\overline{B_2^2}$, $C_1^2\overline{C_2^2}$.
\end{lemma}

\begin{lemma}[{\cite[Lemma~2.3]{apyinvolution}}]\label{lemm2.3apy}
For $n \geq 3$, the group topologically generated by the elements
\begin{align*}
    \{\rho_1, \rho_2, A_1^2\overline{A_2^2}, B_1^2\overline{B_2^2},C_1^2\overline{C_2^2}, h_{1,2}\}
\end{align*}
contains the Dehn twists $A_i^j$, $B_i^j$, $C_{i-1}^j$ for all $1 \leq  j\leq n$ and for all $i \geq 1$.
\end{lemma}

In both of these lemmas, $\rho_1$ and $\rho_2$ are rotations of $\dfrac{2\pi}{n}$ radians, as depicted in Figure \ref{fig:inf}. Note also that $R=\rho_1\rho_2$. We replace $\rho_1,\rho_2$ by $R$ in the previous lemmas to obtain the following:

\begin{lemma}\label{lemm2.2R.apy}
    For $n \geq 3$, the group generated by the elements
\begin{align*}
    \{R, A_1^1\overline{A_1^2}, B_1^1\overline{B_1^2}, C_0^1\overline{C_0^2}, h_{1,2}\}
\end{align*}
contains the Dehn twists $A_1^2\overline{A_2^2}$, $B_1^2\overline{B_2^2}$, $C_1^2\overline{C_2^2}$.
\end{lemma}

\begin{proof}
    Let $G$ be the subgroup generated by the set above.
    
    \textit{Derivation of $B_1^2\overline{B_2^2}$:}
    Observe that the images of the curves $b^1_1$ and $b^2_1$ under the handle shift $h_{1,2}$ are the curves $b^2_1$ and $b^2_2$ respectively. Hence, conjugating $B_1^1\overline{B_1^2}$ by $h_{1,2}$ yields
    \begin{align*}
        &(B_1^1\overline{B_1^2})^{h_{1,2}}=h_{1,2}(B_1^1\overline{B_1^2})\overline{h_{1,2}}=B_1^2\overline{B_2^2} \in G.
    \end{align*}
   
    \textit{Derivation of $A_1^2\overline{A_2^2}$:}
    By conjugating $A_1^1\overline{A_1^2}$ by $R$, we get 
    \begin{align*}
        &(A_1^1\overline{A_1^2})^R = A_1^2\overline{A_1^3} \in G.
    \end{align*}
    Since $b_1^2$ and $a_1^2$ intersect once, applying the braid relation yields  
    \begin{align*}
        (B_1^1\overline{B_1^2})^{B_1^1\overline{B_1^2}\overline{A_1^2\overline{A_1^3}}}=B_1^1\overline{A_1^2} \in G.
    \end{align*}
    Therefore, $B_1^1\overline{A_1^2}$ is in $G$.
    
    It follows that
    \begin{align*}
        &(B_1^1\overline{A_1^2})^{h_{1,2}} = B_1^2\overline{A_2^2} \in G.
    \end{align*}
    Combining these elements yield
    \begin{align*}
        &(A_1^2\overline{B_1^1})(B_1^1\overline{B_1^2})(B_1^2\overline{A_2^2})=A_1^2\overline{A_2^2} \in G.
    \end{align*}

    \textit{Derivation of $C_1^2\overline{C_2^2}$:}
    Conjugation of $C_0^1\overline{C_0^2}$ by $R$ gives 
    \begin{align*}
        &(C_0^1\overline{C_0^2})^R = C_0^2\overline{C_0^3} \in G,\\
        &(C_0^1\overline{C_0^2})(C_0^2\overline{C_0^3}) = C_0^1\overline{C_0^3} \in G.
    \end{align*}
    Conjugating $C_0^1\overline{C_0^3}$ by $h_{1,2}$ gives,
    \begin{align*}
        &(C_0^1\overline{C_0^3})^{h_{1,2}} = C_1^2\overline{C_0^3} \in G.
    \end{align*}

    The element $C_1^2\overline{C_2^2}$ is obtained through use of $h_{1,2}$ and $C_0^1\overline{C_0^3}$:
    \begin{align*}
        &(C_0^1\overline{C_0^3})\overline{(C_1^2\overline{C_0^3})}=(C_0^1\overline{C_0^3})(C_0^3\overline{C_1^2})=C_0^1\overline{C_1^2} \in G,\\
        &(C_0^1\overline{C_1^2})^{h_{1,2}}=C_1^2\overline{C_2^2} \in G.
    \end{align*}
\end{proof}

\begin{lemma}\label{lemm2.3Rapy}
For $n \geq 3$, the group generated by the elements
\begin{align*}
    \{R, A_1^2\overline{A_2^2}, B_1^2\overline{B_2^2},C_1^2\overline{C_2^2}, h_{1,2}\}
\end{align*}
contains the Dehn twists $A_i^j$, $B_i^j$, $C_{i-1}^j$ for all $1 \leq  j\leq n$ and for all $i \geq 1$.
\end{lemma}
\begin{proof}
    The proof is the same as that of Lemma~\ref{lemm2.3apy} in~\cite{apyinvolution}.
\end{proof}

Using these intermediate results, we prove the following theorem, which will aid us to prove that our generating sets indeed topologically generate $\mod(S(n))$.

\begin{theorem}\label{thm:23}
    For $n \geq 3$, the group generated by the elements 
   \begin{align*}
    \{R, A_1^1\overline{A_1^2}, B_1^1\overline{B_1^2}, C_0^1\overline{C_0^2}, h_{1,2}\}
    \end{align*}
contains the Dehn twists $A_i^j$, $B_i^j$, $C_{i-1}^j$ for all $j=1,2, \ldots, n$ and for all $i=1,2,3, \ldots$.
\end{theorem}

\begin{proof}
    Let $G$ be the subgroup generated by the above set. By Lemma~\ref{lemm2.2R.apy}, $G$ contains $A_1^2\overline{A_2^2}$, $B_1^2\overline{B_2^2}$, $C_1^2\overline{C_2^2}$, which implies by Lemma~\ref{lemm2.3Rapy}, that it also contains $A_i^j$, $B_i^j$, $C_{i-1}^j$.
\end{proof}

We also require the following modified form of Theorem~\ref{thm:23} to prove Theorem~\ref{thm:first}. This is because in its proof, we will not be able to isolate the handle shift $h_{1,2}$ from the Dehn twists. Instead, we need to generate the Dehn twists first using $h_{1,2}\overline{h_{5,6}}$, then isolate $h_{1,2}$ to generate the whole group.

\begin{theorem}\label{thmnew}
    For $n \geq 6$, the group generated by the elements
    \begin{align*}
    \{R, A_1^1\overline{A_1^2}, B_1^1\overline{B_1^2}, C_0^1\overline{C_0^2}, h_{1,2}\overline{h_{5,6}}\}
    \end{align*}
    contains the Dehn twist $A_i^j,B_i^j,C_{i-1}^j$ for all $j=1,2,\ldots, n$ and for all $i=1,2,3, \ldots$.
\end{theorem}
\begin{proof}
The proof follows exactly the same as those of Lemma~\ref{lemm2.2R.apy} and Lemma~\ref{lemm2.3Rapy}, with the exception that $h_{1,2}$ is replaced by $h_{1,2}\overline{h_{5,6}}$. Note that the addition of the $\overline{h_{5,6}}$ factor has no effect on any of the steps as the support of it is disjoint from those of the Dehn twists used in those proofs and the support of $h_{1,2}$.  
\end{proof}

We are ready to state and prove the first main result of this paper.

\begin{theorem}\label{thm:first}
For $n \geq 8$, $\mod(S(n))$ is topologically generated by three elements.
\end{theorem}

\begin{proof}

Let $G$ be the group topologically generated by the three elements $R$, $\tau$, and $F_1 = B_1^1C_0^3A_1^{4}h_{n-1,n}$, where $R$ is a counterclockwise rotation of $\frac{2\pi}{n}$ radians and $\tau$ is a homeomorphism of $S(n)$ that swaps two of the ends while fixing the others, see~\cite{apyinvolution,tez} for an example of such a homeomorphism. It follows that $R$ projects to the $n$-cycle $(1\,2\,\dots\,n)$ and $\tau$ projects to the $2$-cycle $(1\,2)$ inside the symmetric group $\Sym_n$. It is a known fact that the $n$-cycle $(1\,2\, \dots \,n)$ and the $2$-cycle $(1\, 2)$ generate $\Sym_n$ and thus, elements of $G$ cover $\Sym_n$. Our goal is to show that $G$ contains the Dehn twists $A_i^j,B_i^j,C_{i-1}^j$ for all $j=1,2,\ldots, n$ and for all $i=1,2,3, \ldots$ using Theorem~\ref{thmnew}, and the handle shift $h_{1,2}$, which implies by [\citenum{tez}, Proposition 6.1.10] that it contains $\overline{\pmod_\mathrm{c}(S(n))}$.

\begin{figure}[H]
      \centering
      \includegraphics[width=0.60\textwidth]{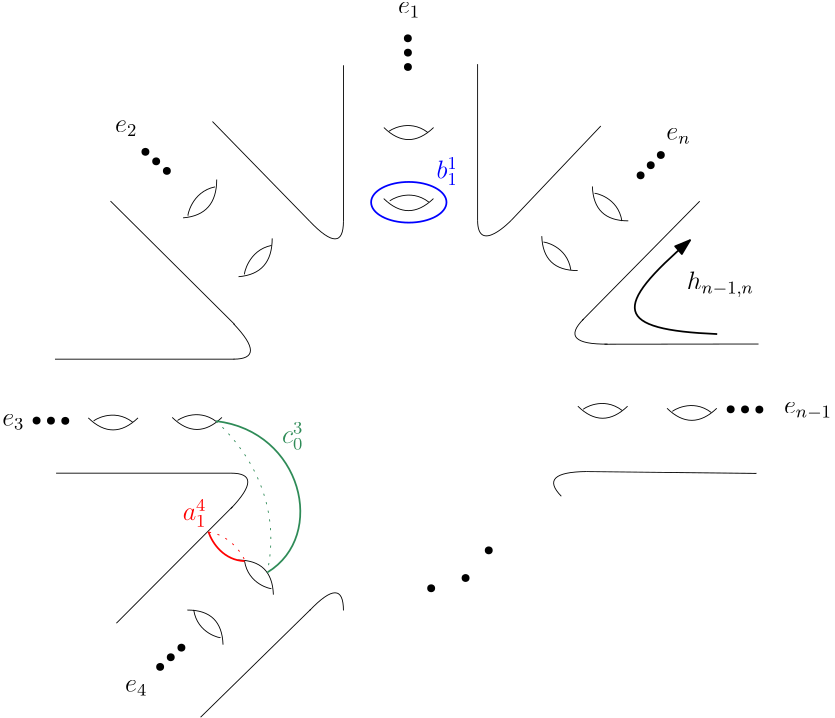}
      \caption{The curves corresponding to the Dehn twist factors of $F_1=B_1^1C_0^3A_1^{4}h_{n-1, n}$.}
      \label{fig:n8}
\end{figure}

To simplify notation, we use the simplified notation described in the introduction, where the upper (end) index of a Dehn twist is written as a subscript, while the lower (genus) index is assumed to be 1 (or 0 for $C$ curves). Also, we simplify the notation for handle shifts where $h_{n-1, n}$ is denoted $h_{n-1}$. With this notation, our second generator is:
\begin{align*}
    F_1 = B_1C_3A_4h_{n-1}.
\end{align*}
First, we generate a second element, $F_2$, by conjugating $F_1$ by $R^2$. Since $R$ increments the end index by one, $R^2$ increments it by two:
\begin{align*}
    F_2 = F_1^{R^2} = R^2 F_1 \overline{R}^2 = B_3 C_5 A_6 h_1 \in G.
\end{align*}

Next, we generate new elements through a series of conjugations: 
\begin{align*}
    F_3 &= F_1^{F_1F_2} = B_2 B_3 A_4 h_{n-1}, \\
    F_4 &= F_3^{F_3F_1} = B_2 C_3 A_4 h_{n-1}.
\end{align*}
Taking the product of $F_3$ and $\overline{F_4}$ cancels common terms:
\begin{align*}
    F_3\overline{F_4} = (B_2 B_3 A_4 h_{n-1}) (\overline{h_{n-1}}\mkern3mu\overline{A_4}\mkern3mu\overline{C_3}\mkern3mu\overline{B_2}) = B_2 B_3 \overline{C_3}\mkern3mu\overline{B_2} = B_3\overline{C_3} \in G,
\end{align*}
where the final step uses the fact that mapping classes on different ends commute. Similarly:
\begin{align*}
    F_1\overline{F_4} = (B_1 C_3 A_4 h_{n-1}) (\overline{h_{n-1}}\mkern3mu\overline{A_4}\mkern3mu \overline{C_3}\mkern3mu\overline{B_2}) = B_1 C_3 \overline{C_3}\mkern3mu\overline{B_2} = B_1\overline{B_2} \in G.
\end{align*}
From these elements, we can generate a family of similar elements using the rotation $R$:
\begin{align*}
    (B_1\overline{B_2})^R &= B_2\overline{B_3} \in G,\\
    (B_1\overline{B_2})^{R^2} &= B_3\overline{B_4}\in G,\\
    (B_3\overline{C_3})^R &= B_4\overline{C_4} \in G.
\end{align*}
By combining these, we obtain products of Dehn twists on non-adjacent curves:
\begin{align*}
    (C_3\overline{B_3})(B_3\overline{B_4})(B_4\overline{C_4})&=C_3\overline{C_4} \in G,\\
    (C_3\overline{C_4})^{\overline{R}^2}&=C_1\overline{C_2} \in G.
\end{align*}
Also, since $B_1\overline{B_2} \in G$ and $B_2\overline{B_3} \in G$, their product $B_1\overline{C_3} = (B_1\overline{B_2})(B_2\overline{B_3})(B_3\overline{C_3})$ is in $G$. We use this to generate $A_1\overline{A_2}$. Define:
\begin{align*}
    F_5 = (B_1\overline{C_3})F_1 = (B_1\overline{C_3})(B_1C_3A_4h_{n-1}) = B_1^2 A_4 h_{n-1}\in G.
\end{align*}
Then we rotate this element and its inverse:
\begin{align*}
    F_5^{\overline{R}^3} &= B_{n-2}^2 A_1 h_{n-4}, \\
    \overline{F_5}^{\overline{R}^2} &= \overline{h_{n-3}}\mkern3mu\overline{A_2}\mkern3mu\overline{B_{n-1}}^2\in G.
\end{align*}
The following conjugations  yield:
\begin{align*}
    (B_{1}\overline{B_{2}})^{(B_{1}\overline{B_{2}})F_5^{\overline{R}^3}} &= A_{1}\overline{B_{2}} \in G, \\
    (A_{1}\overline{B_{2}})^{(A_{1}\overline{B_{2}})\overline{F_5}^{\overline{R}^2}} &= A_{1}\overline{A_{2}} \in G.
\end{align*}
Finally, we isolate a composite handle shift. Define:
\begin{align*}
    F_6 &= F_1^{R^4} = B_5 C_7 A_8 h_3, \in G, \\
    F_7 &= F_1\overline{F_6} = (B_1C_3A_4h_{n-1})(\overline{h_3}\mkern3mu\overline{A_8}\mkern3mu\overline{C_7}\mkern3mu\overline{B_5}) = (B_1\overline{B_5})(C_3\overline{C_7})(A_4\overline{A_8})(h_{n-1}\overline{h_3})\in G. \\
\end{align*}
Observe that 
\begin{align*}
    & (B_1\overline{B_2})(B_2\overline{B_3})(B_3\overline{B_4})(B_4\overline{B_5}) = B_1\overline{B_5} \in G, \\
    & (C_3\overline{C_4})(C_4\overline{C_5})(C_5\overline{C_6})(C_6\overline{C_7}) = C_3\overline{C_7} \in G, \\
    & (A_4\overline{A_5})(A_5\overline{A_6})(A_6\overline{A_7})(A_7\overline{A_8}) = A_4\overline{A_8} \in G.
\end{align*}
Multiplying $F_7$ by the inverses of these, we isolate $h_{n-1}\overline{h_3} \in G$. Conjugating this by $R^2$ gives:
\begin{align*}
    (h_{n-1}\overline{h_3})^{R^2} = h_1\overline{h_5} \in G.
\end{align*}
We have shown that $G$ contains $R$, $A_1\overline{A_2}$, $B_1\overline{B_2}$, $C_1\overline{C_2}$, and $h_1\overline{h_5}$. By Theorem~\ref{thmnew}, the subgroup generated by them contain all Dehn twists $A_i^j, B_i^j, C_{i-1}^j$ and thus contains $\overline{\pmod_\mathrm{c}(S(n))}$. With all individual Dehn twists now in $G$, we can return to the definition of $F_1$:
\begin{align*}
    h_{n-1} = (\overline{A_4}\mkern3mu\overline{C_3}\mkern3mu\overline{B_1}) F_1 \in G.
\end{align*}
Since $h_{n-1}$ and $R$ are in $G$, all handle shifts $h_{j,j+1}$ can be generated. It follows that $G$ is $\mod(S(n))$.
\end{proof}

\begin{theorem}\label{thm:second}
For $n \geq 3$, $\mod(S(n))$ is topologically generated by four elements.
\end{theorem}
\begin{proof}
    Let $G$ be the subgroup topologically generated by $\{R, \tau, F_1, h_{1,2}\}$, where
    \begin{align*}
        F_1=A_1^{n-1}B_1^nC_0^{n-2},
    \end{align*}
    and $\tau$ is as the previous theorem.
    Conjugating $F_1$ by $h_{1,2}R^2$, we get
    \begin{align*}
        &F_2 = F_1^{h_{1,2}R^2}= h_{1,2}R^2F_1\overline{R^2}\mkern3mu\overline{h_{1,2}}={A'}_1^2B_2^2C_0^n.
    \end{align*}
    Also, conjugating $F_1$ by $F_1F_2$ yields
    \begin{align*}
        &F_3 = F_1^{F_1F_2}=A_1^{n-1}C_0^{n}C_0^{n-2}.
    \end{align*}
    Then,
    \begin{align*}
        &F_1\overline{F_3}=B_1^n\overline{C_0^n} \in G.
    \end{align*}
  Note that, $h_{1,2}^{\overline{R}}=h_{n,1} \in G$. Using the conjugation of $B_1^n\overline{C_0^n}$ by $\overline{R}h_{n,1}$, we have
    \begin{align*}
        &(B_1^n\overline{C_0^n})^{\overline{R}h_{n,1}}=B_1^n\overline{C_1^n}\in G,\\
        &(C_0^n\overline{B_1^n})(B_1^n\overline{C_1^n}) = C_0^n\overline{C_1^n} \in G,\\
        &(C_0^n\overline{C_1^n})^{\overline{R}h_{n,1}}=C_1^n\overline{C_0^{n-1}} \in G,\\
        &(C_0^{n-1}\overline{C_1^n})(C_1^n\overline{C_0^n})=C_0^{n-1}\overline{C_0^n}\in G,\\
        &(C_0^{n-1}\overline{C_0^n})^{R^{2-n}}=C_0^1\overline{C_0^2}\in G.
    \end{align*}
    Also, since $B_1^n\overline{C_0^n}, C_0^{n-1}\overline{C_0^n} \in G$,
    \begin{align*}
        &(B_1^n\overline{C_0^n})^{R}=B_1^1\overline{C_0^{1}} \in G,\\
        &(C_0^{n-1}\overline{C_0^n})^{R}=C_0^{n}\overline{C_0^1} \in G,\\
        &(B_1^n\overline{C_0^n})(C_0^{n}\overline{C_0^1})(C_0^{1}\overline{B_1^1})=B_1^n\overline{B_1^1}\in G,\\
        &(B_1^n\overline{B_1^1})^R=B_1^1\overline{B_1^2} \in G.
    \end{align*}
    Finally, since 
    \begin{align*}
        &B_1^3\overline{C_0^{n}}=B_1^3\overline{B_1^2}B_1^2\overline{B_1^1}B_1^1\overline{B_1^n}B_1^n\overline{C_0^{n}} \in G,\\
        &F_4=B_1^3\overline{C_0^{n}}F_1^{R^2}B_1^3\overline{B_1^2}=A_1^1B_1^3B_1^3 \in G,\\
    \end{align*}
    then
    \begin{align*}
        &(B_1^1\overline{B_1^2})^{(B_1^1\overline{B_1^2}F_4)}=A_1^1\overline{B_1^2} \in G,\\
        &(A_1^1\overline{B_1^2})(B_1^2\overline{B_1^1})=A_1^1\overline{B_1^1} \in G,\\
        &(A_1^1\overline{B_1^1})(B_1^1\overline{B_1^2})\overline{(A_1^1\overline{B_1^1})^R}=A_1^1\overline{B_1^1}B_1^1\overline{B_1^2}B_1^2\overline{A_1^2}=A_1^1\overline{A_1^2} \in G.
    \end{align*}
    (Note that by conjugating $A_1^1\overline{A_1^2}$ by $h_{1,2}$ and $R$, one can get $(A_1^1)'\overline{(A_1^2)'} \in G$.)\\
    We have shown that $A_1^1\overline{A_1^2}$, $B_1^1\overline{B_1^2}$, $C_0^1\overline{C_0^2}$, $R$, $\tau$, and $h_{1,2}$ are in $G$. It immediately follows by the same argument as in the proof of the previous theorem (except for the fact that we use Theorem~\ref{thm:23} instead of Theorem~\ref{thmnew}) that $G$ is $\mod(S(n))$.
    
\end{proof}
\subsection{The Jacob's Ladder surface}
For the Jacob's Ladder surface, we use the following model.

\begin{figure}[H]
      \centering
      \includegraphics[width=0.5\textwidth]{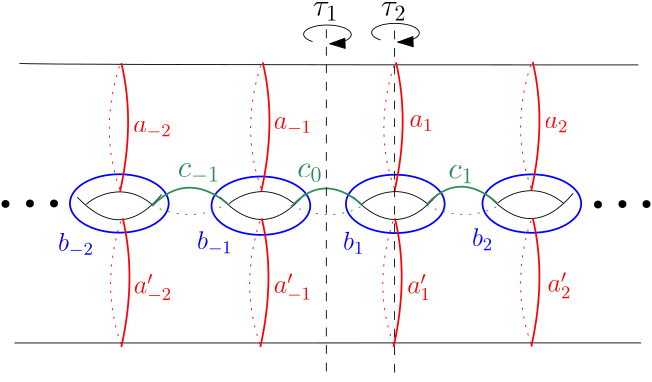}
      \caption{ A model for the Jacob's Ladder surface, showing the indexed curves and the rotational symmetries
      $\tau_1$ and $\tau_2$.}
      \label{fig:curves}
\end{figure}
Observe that
\begin{align*}
    H= \tau_2\tau_1
\end{align*}
is a handle shift.

Throughout this subsection, the Jacob's Ladder surface is denoted by $S$.
\begin{lemma}\label{lem:Jacob}
    The subgroup of $\mod(S)$, generated by 
    \begin{align*}
    \{H,A_1\overline{A_2},A'_1\overline{A'_2},B_1\overline{B_2},C_1\overline{C_2} \}
    \end{align*}
    contains the Dehn twists $A_i,A'_i,B_i$ and $C_j$ for all $|i| \geq 1$ $j\in \mathbb{Z}$.
\end{lemma}
\begin{proof}
    The proof essentially follows the same steps of Lemma~\ref{lemm2.3apy}, except that we don't actually need the rotation $R$, since $S$ has two ends. Let $G$ be the subgroup topologically generated by the above set where $H$ is the handle shift where $-\infty$ is the repelling and  $+\infty$ is the attracting end.
    By conjugating $A_1\overline{A_2}$ with $H$, we get $A_2\overline{A_3}$. Then,
    \begin{align*}
        (A_1\overline{A_2})(A_2\overline{A_3}) = A_1\overline{A_3} \in G.
    \end{align*}
    Note that the curves $a_1$ and $b_1$ intersect once, so by the braid relation
    \begin{align*}
        (A_1\overline{A_2})(B_1\overline{B_2})(a_1,a_3) = (b_1,a_3).
    \end{align*}
    We conjugate $A_1\overline{A_3}$ by $A_1\overline{A_2}B_1\overline{B_2}$ use the conjugation property of Dehn twists to get
    \begin{align*}
        (A_1\overline{A_2}B_1\overline{B_2})(A_1\overline{A_3})(\overline{B_1\overline{B_2}A_1\overline{A_2})} = B_1\overline{A_3} \in G.
    \end{align*}
    If we start with $C_1\overline{C_2}$ instead of $A_1\overline{A_2}$ and apply the same argument, we get $C_1\overline{A_3} \in G$. Then,
    \begin{align*}
        (B_2\overline{B_1})(B_1\overline{A_3})(A_3\overline{A_2})(A_2\overline{A_1}) = B_2\overline{A_1} \in G, 
    \end{align*}
    \begin{align*}
            (C_1\overline{A_3})(A_3\overline{A_2})(A_2\overline{A_1}) = C_1\overline{A_1} \in G,
    \end{align*}
    \begin{align*}
      (C_2\overline{C_1})(C_1\overline{A_1}) = C_2\overline{A_1} \in G,
    \end{align*}
    \begin{align*}
    (A_2\overline{A_1})(A_1\overline{C_2}) = A_2\overline{C_2} \in G.
    \end{align*}
    Consider the embedded lantern bounded by the curves $a_1,c_1,c_2$ and $a_3$. We have,
    \begin{align*}
    (B_2\overline{A_1})(C_1\overline{A_1})(A_1\overline{A_2})(C_1\overline{A_2})(b_2,a_1) = (d_1,a_1),
    \end{align*}
    and by conjugation we conclude that $D_1\overline{A_1} \in G$. Conjugating $B_1\overline{B_2}$ by $H$, we get that $B_2\overline{B_3}$ and $B_3\overline{A_1}$ is in $G$. Using the lantern relation, 
    \begin{align*}
    (B_3\overline{A_1})(C_2\overline{A_1})(A_3\overline{A_1})(B_3\overline{A_1})(d_1,a_1) = (d_2,a_1),
    \end{align*}
    so by conjugation, $D_2\overline{A_1}\in G$.
    
\begin{figure}[H]
    \centering
    \begin{minipage}[b]{0.45\textwidth}  
        \centering
        \adjustbox{valign=m}{\includegraphics[width=\linewidth, height=4cm, keepaspectratio]{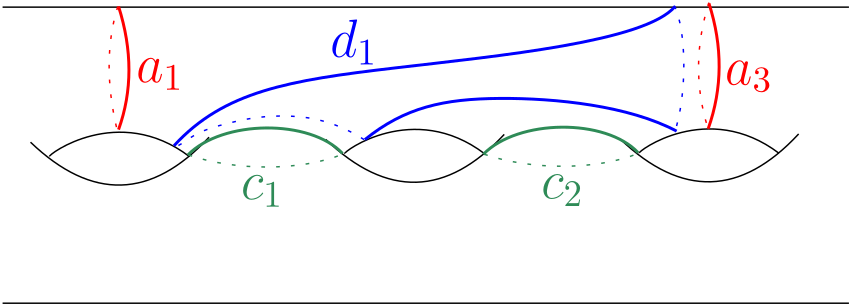}}
    \end{minipage}
    \hfill 
    \begin{minipage}[b]{0.45\textwidth}  
        \centering
        \adjustbox{valign=m}{\includegraphics[width=\linewidth, height=4cm, keepaspectratio]{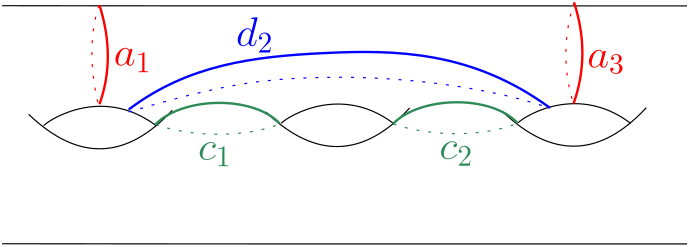}}
    \end{minipage}
    \caption{The curves $d_1$ and $d_2$.}
    \label{fig:d1d2}
\end{figure}

    Finally, 
    \begin{align*}
    (D_2\overline{A_1})(A_1\overline{C_1}) = D_2\overline{C_1} \in G.
    \end{align*}
    By the lantern relation,
    \begin{align*}
    A_1C_1C_2A_3 = A_2D_1D_2,
    \end{align*}
    and thus
    \begin{align*}
    A_3 = (A_2\overline{C_2})(D_1\overline{A_1})(D_2\overline{C_1}) \in G.
    \end{align*}
    Conjugating with powers of $H$, we have that $A_i \in G$ for any $|i| \geq 1$. It immediately follows that every $B_i$ and $C_j$ is also in $G$. Note that by following the same procedures for $A'_i$ instead of $A_i$ we get that $A'_i \in G$ and we are done.
\end{proof}
\begin{theorem}\label{thm:third}
The mapping class group of the Jacob's Ladder surface is topologically generated by three elements. 
\end{theorem}
\begin{proof}
    Let $\tau_1,\tau_2$ be the rotations as shown in Figure~\ref{fig:curves}, and let $G$ be the subgroup topologically generated by the set $\{\tau_1,\tau_2,A_1A'_6C_1B_3\overline{B_{11}}\mkern3mu\overline{C_{12}}\mkern3mu\overline{A'_8}\mkern3mu\overline{A_{13}}\}$.
    Let 
    \begin{align*}
F_1=A_1A'_6C_1B_3\overline{B_{11}}\mkern3mu\overline{C_{12}}\mkern3mu\overline{A'_8}\mkern3mu\overline{A_{13}}.
    \end{align*}

    Note that $\tau_2\tau_1 = H \in G$.
    \begin{figure}[H]
      \centering
      \includegraphics[width=0.90\textwidth]{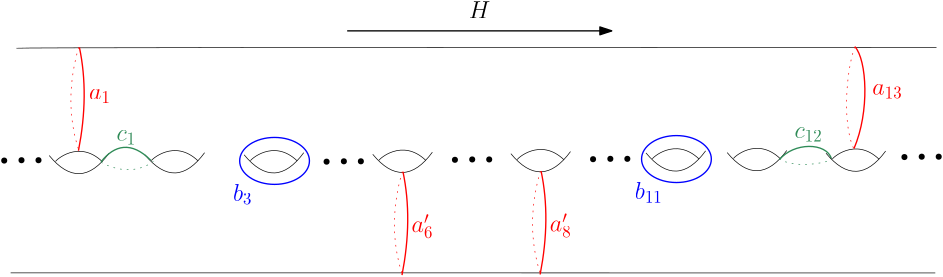}
      \caption{The curves corresponding to the Dehn twist factors of $F_1=A_1A'_6C_1B_3\overline{B_{11}}\mkern3mu\overline{C_{12}}\mkern3mu\overline{A'_8}\mkern3mu\overline{A_{13}}$.}
      \label{fig:n2curves}
    \end{figure}
    Since $\tau_2\tau_1 = H \in G$, we get 
    \begin{align*}
    F_2 = F_1^H = A_2A'_7C_2B_4\overline{B_{12}}\mkern3mu\overline{C_{13}}\mkern3mu\overline{A'_9}\mkern3mu\overline{A_{14}} \in G.
    \end{align*}
    Since $c_2$ intersect once with $b_3$ and $c_{12}$ intersects once with $b_{12}$, we have
    \begin{align*}
    F_3 = F_2^{F_2F_1} = A_2A'_7B_3B_4\overline{C_{12}}\mkern3mu\overline{C_{13}}\mkern3mu\overline{A'_9}\mkern3mu\overline{A_{14}} \in G.
    \end{align*}
    We conjugate $F_3$ by $\overline{H}$ to get 
    \[
    F_4 = F_3^{\overline{H}}=A_1A'_6B_2B_3\overline{C_{11}}\mkern3mu\overline{C_{12}}\mkern3mu\overline{A'_8}\mkern3mu\overline{A_{13}} \in G,
    \]
    and then since $a_2$ intersects with $b_2$ once,
    \[
    F_5 = F_4^{F_4F_3} = A_1A_2A'_6B_3\overline{C_{11}}\mkern3mu\overline{C_{12}}\mkern3mu\overline{A'_8}\mkern3mu\overline{A_{13}} \in G,
    \]
    and we get $F_4\overline{F_5} = B_2\overline{A_2} \in G$ which imply $A_i\overline{B_i}\in G$ for all $i\in \mathbb{Z}\setminus \{0\}$. Next, note that 
    \[
    F_6 = (A_3\overline{B_3})F_1 = A_1A_3A'_6C_1\overline{B_{11}}\mkern3mu\overline{C_{12}}\mkern3mu\overline{A'_8}\mkern3mu\overline{A_{13}} \in G.
    \]
    Then,
    \[
    F_7 = F_6^H = A_2A_4A'_7C_2\overline{B_{12}}\mkern3mu\overline{C_{13}}\mkern3mu\overline{A'_9}\mkern3mu\overline{A_{14}} \in G. 
    \]
    Since $b_{12}$ intersects once with $c_{12}$,
    \[
    F_8 = F_7^{F_7F_6} = A_2A_4A'_7C_2\overline{C_{12}}\mkern3mu\overline{C_{13}}\mkern3mu\overline{A'_9}\mkern3mu\overline{A_{14}} \in G.
    \]
    Now we have $\overline{F_7}F_8 = B_{12}\overline{C_{12}} \in G$ which implies that $B_i\overline{C_i}$ for $i\geq 1$. Note that $\tau_1B_{1}\overline{C_1}\tau_1 = B_{-1}\overline{C_{-1}}$. Since $H^2$ sends $b_{-1}$ to $b_2$ and $c_{-1}$ to $c_1$, we have $B_{2}\overline{C_1}\in G$ and thus
    \[
    (B_1\overline{C_1})(C_1\overline{B_2}) = B_1\overline{B_2} \in G.
    \]
    Let 
    \begin{align*}
    F_9 &= (B_1\overline{A_1})(B_2\overline{C_1})F_1(B_{11}\overline{C_{11}})\\
    &=A'_6B_1B_2B_3\overline{C_{11}}\mkern3mu\overline{C_{12}}\mkern3mu\overline{A'_8}\mkern3mu\overline{A_{13}}\in G.
    \end{align*}
    Then 
    \[
    F_{10} = F_9^{H^3} = A'_9B_4B_5B_6\overline{C_{14}}\mkern3mu\overline{C_{15}}\mkern3mu\overline{A'_{11}}\mkern3mu\overline{A_{16}}\in G.
    \]
    Note that $a'_6$ intersects once with $b_6$. It follows that
    \[
    F_{11} = F_{10}^{F_{10}F_{9}} =A'_6A'_9B_4B_5\overline{C_{14}}\mkern3mu\overline{C_{15}}\mkern3mu\overline{A'_{11}}\mkern3mu\overline{A_{16}} \in G.
    \]
    Now we have $F_{11}\overline{F_{10}} = A'_6\overline{B_6}\in G$ which implies that $A'_i\overline{B_i}\in G$. Finally,
    \[
    (A_1\overline{B_1})(B_1\overline{B_2})(B_2\overline{A_2}) = A_1\overline{A_2}\in G,
    \]
    \[
    (C_1\overline{B_1})(B_1\overline{B_2})(B_2\overline{C_2}) = C_1\overline{C_2}\in G,
    \]
    \[
    (A'_1\overline{B_1})(B_1\overline{B_2})(B_2\overline{A'_2}) = A'_1\overline{A'_2} \in G,
    \]
    and by Lemma~\ref{lem:Jacob}, $G$ contains the Dehn twists $A_i,A'_i,B_i$ and $C_j$ for all $|i| \geq 1$ $j\in \mathbb{Z}$. By [\citenum{tez}, Proposition 6.1.18], $G$ contains $\overline{\pmod_\mathrm{c}(S)}$. Since the action of $\tau_1$ on the space of ends generate $\Sym_2$, and $H$ is in $G$, $G$ must be $\mod(S(n))$.
    
\end{proof}
\subsection{The Loch Ness Monster surface}
\vspace{2mm}

Throughout this subsection, the Loch Ness Monster surface is denoted by $S$. To establish our minimal generating set for the Loch Ness Monster surface, we use a different model for it.

\begin{proposition}[{\cite[Proposition 6.1.15]{tez}}]\label{prop:loch}
     The mapping class group of the Loch Ness Monster surface is topologically generated by the set of Dehn twists about the curves $\{a_i, b_i, c_j \mid i \in \mathbb{Z}\setminus\{0\}, j \in \mathbb{Z}\}$, as shown in Figure~\ref{fig:newmodelofn1}.
\end{proposition}
\begin{figure}[H]
      \centering
      \includegraphics[width=0.45\textwidth]{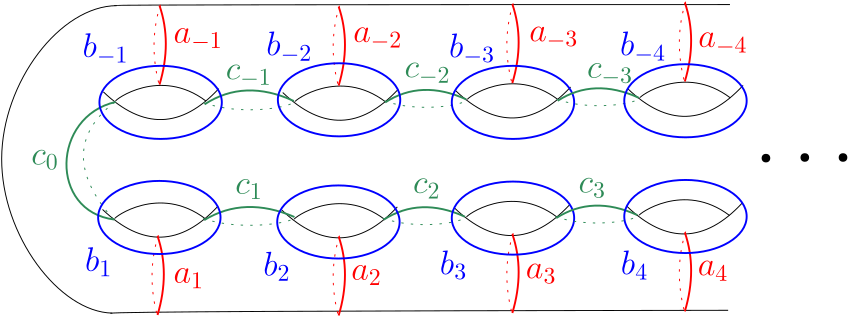}
      \caption{The generating curves for the Loch Ness Monster surface. }
      \label{fig:newmodelofn1}
   \end{figure}

\begin{lemma}\label{lem:a1a2}
    The subgroup of $\mod(S)$ topologically generated by 
    \[\{H,A_1\overline{A_2},B_1\overline{B_2},C_1\overline{C_2} \}\]
    contains the Dehn twists $A_i,B_i$ and $C_j$ for all $|i| \geq 1$ and $j\in \mathbb{Z}$.
\end{lemma}
\begin{proof}
    The proof is essentially the same with Lemma~\ref{lem:Jacob} except for the fact that the a lantern is embedded as seen in Figure~\ref{fig:ld1d2}.
    
\begin{figure}[H]
    \centering
    \begin{minipage}[b]{0.45\textwidth}  
        \centering
        \adjustbox{valign=m}{\includegraphics[width=\linewidth, height=4cm, keepaspectratio]{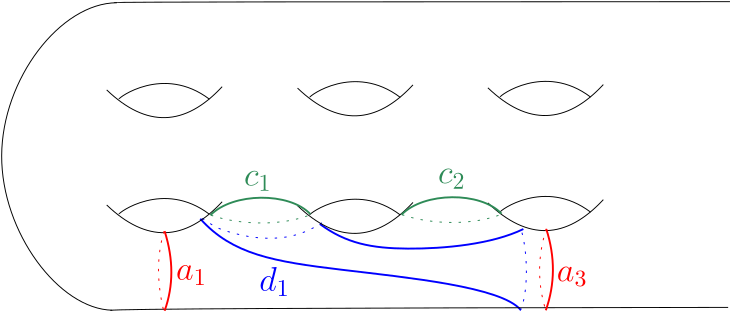}}
    \end{minipage}
    \hfill 
    \begin{minipage}[b]{0.45\textwidth}  
        \centering
        \adjustbox{valign=m}{\includegraphics[width=\linewidth, height=4cm, keepaspectratio]{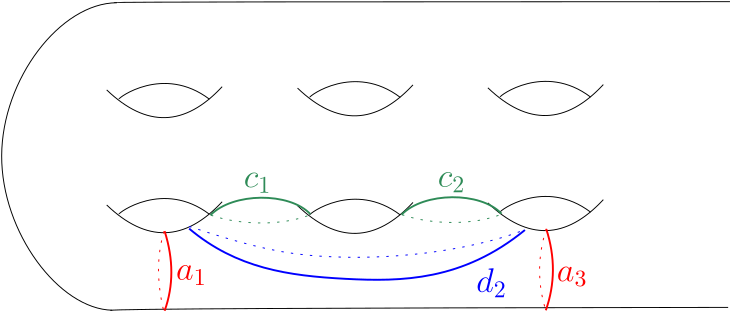}}
    \end{minipage}
    \caption{The curves $d_1$ and $d_2$.}
    \label{fig:ld1d2}
\end{figure}
\end{proof}

\begin{theorem}\label{thm:fourth}
    The mapping class group of the Loch Ness Monster surface is topologically generated by two elements.
\end{theorem}
\begin{proof}
 Let $G$ be the subgroup topologically generated by the set $\{H,F_1\}$, where $H$ is a handle shift as shown in Figure~\ref{fig:n1curves}, and $F_1 = A_4C_0B_{-2}$. 
 
\begin{figure}[H]
      \centering
      \includegraphics[width=0.55\textwidth]{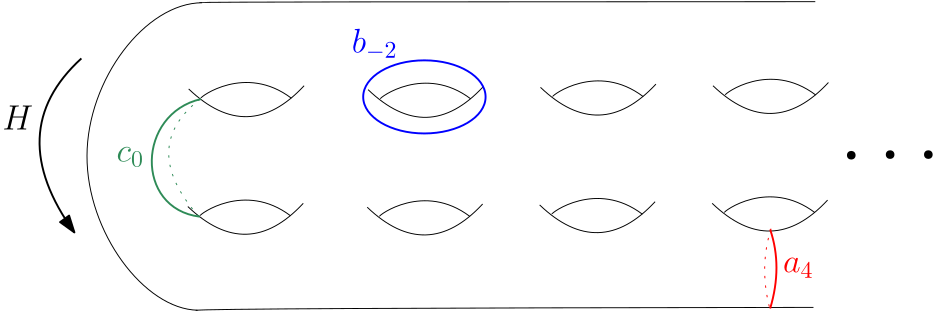}
      \caption{The disjointly supported curves involved in $F_1= A_4C_0B_{-2}$.}
      \label{fig:n1curves}
\end{figure}

Conjugating $F_1$ by $H$, we get
    \[
    F_2 = F_1^H =A_5C_1B_{-1}\in G.
    \]
Note that the curves $c_0$ and $b_{-1}$ intersect once. By the braid relation, 
    \begin{align*}
        F_3 = F_2^{F_2F_1} &= (A_5C_1B_{-1})(A_4C_0B_{-2})(A_5C_1B_{-1})(\overline{B_{-2}}\mkern3mu\overline{C_0}\mkern3mu\overline{A_4})(\overline{B_{-1}}\mkern3mu\overline{C_1}\mkern3mu \overline{A_5})\\
        &=(A_5C_1B_{-1})C_0(C_1B_{-1})\overline{C_0}(\overline{B_{-1}}\mkern3mu\overline{C_1}\mkern3mu) \\
        &=A_5C_0C_1\in G.
    \end{align*}
    Then, $F_2\overline{F_3} = B_{-1}\overline{C_0}$ is in $G$, which implies that $B_2\overline{C_2}$ is also in $G$, through conjugation by $H^2$. Next,
    \begin{align*}
        F_4 = (B_2\overline{C_2})(F_1^{H^2}) = (B_2\overline{C_2})A_6C_2B_1 = A_6B_1B_2 \in G.
    \end{align*}
    Since $c_0$ and $b_1$ intersect once,
    \begin{align*}
        F_5 = F_4^{F_4F_1} &= B_1C_0(A_6B_1B_2)\overline{C_0}\mkern3mu\overline{B_1}\\
        &= A_6C_0B_2 \in G.
    \end{align*}
    Note that $F_4\overline{F_5} = B_1\overline{C_0} \in G$, and $(B_{-1}\overline{C_0})(C_0\overline{B_1}) = B_{-1}\overline{B_1}\in G$ implies $B_1\overline{B_2}\in G$. 
    Conjugate $F_1$ with $\overline{H^5}$ so that 
    \begin{align*}
        F_6 = F_1^{\overline{H^5}} = A_{-2}C_{-5}B_{-7}.
    \end{align*}
    Then,
    \[
    F_7 = (B_{-6}\overline{C_{-5}})A_{-2}C_{-5}B_{-7} = A_{-2}B_{-6}B_{-7} \in G.
    \]
    Since $a_{-2}$ and $b_{-2}$ intersect once, 
    \begin{align*}
        F_8 = F_7^{F_7F_1} &= A_{-2}B_{-2}(A_{-2}B_{-6}B_{-7})\overline{B_{-2}}\mkern3mu \overline{A_{-2}}\\
        &= B_{-2}B_{-6}B_{-7} \in G.
    \end{align*}
    Then, $F_7\overline{F_8} = A_{-2}\overline{B_{-2}}$. Note also that $(A_{-1}\overline{B_{-1}})(B_{-1}\overline{C_{-1}}) =A_{-1}\overline{C_{-1}} \in G.$

    Next we consider the action of $(C_{-1}\overline{A_{-1}})F_1$ on the pair of curves $(a_{-1},c_{-1})$. Since $b_{-2}$ and $c_{-1}$ intersect once, we see that 
    \[
        (C_{-1}\overline{A_{-1}})F_1(a_{-1},c_{-1}) = (C_{-1}\overline{A_{-1}})A_4C_0B_{-2}(a_{-1},c_{-1}) = (a_{-1},b_{-2}).
    \]
    We conjugate $A_{-1}\overline{C_{-1}}$ with $(C_{-1}\overline{A_{-1}})F_1$
    \begin{align*}
        F_9 &= (C_{-1}\overline{A_{-1}}A_4C_0B_{-2})(A_{-1}\overline{C_{-1}})\overline{(C_{-1}\overline{A_{-1}}A_4C_0B_{-2})}\\
        &=(C_{-1}B_{-2})(A_{-1}\overline{C_{-1}})\overline{B_{-2}}\mkern3mu\overline{C_{-1}}\\
        &=A_{-1}\overline{B_{-2}} \in G.
    \end{align*}
    Then,
    \begin{align*}
        &(A_{-2}\overline{B_{-2}})(B_{-2}\overline{A_1}) = A_{-2}\overline{A_ {-1}} \in G,\\
        &(A_{-2}\overline{A_{-1}})^{H^2}=A_1\overline{A_2} \in G.
    \end{align*}
    We now have $A_1\overline{A_2}\in G$. Finally,
    \begin{align*}
    &(C_0\overline{B_{-1}})(B_{-1}\overline{B_1})(B_1\overline{C_1}) = C_0\overline{C_1} \in G,\\
    &(C_0\overline{C_1})^H=C_1\overline{C_2} \in G.
    \end{align*}
    It follows by Lemma~\ref{lem:a1a2} and Proposition~\ref{prop:loch} that $G$ is $\mod(S(n))$.
\end{proof}

\bibliographystyle{amsplain}

\bibliography{references.bib}

\end{document}